\documentclass[12pt]{amsart}
\usepackage{fullpage}

\usepackage{mathtools}

\usepackage{amssymb,amsmath,amsthm}
\usepackage[english]{babel}
\usepackage{tikz-cd}
\usepackage[colorlinks=true, citecolor=olive, filecolor=olive, linkcolor=teal,pagebackref, hyperindex]{hyperref}

\newcommand{\height}{\operatorname{ht}}
\newcommand{\length}{\operatorname{\lambda}}

\newcommand{\rank}{\operatorname {rank}}
\newcommand{\Spec}[1]{\operatorname {Spec(#1)}}
\newcommand{\mf}[1]{\mathfrak #1}

\DeclareMathOperator{\eh}{e}
\DeclareMathOperator{\ehk}{e_{HK}}

\DeclareMathOperator{\Hom}{Hom}
\DeclareMathOperator{\fsig}{s}

\DeclareMathOperator{\srk}{rank}
\DeclareMathOperator{\depth}{depth}

\renewcommand{\frq}[1]{{#1}^{[p^e]}}
\renewcommand{\hat}{\widehat}
\newcommand{\fa}{\mathfrak a}

\newcommand{\fm}{\mathfrak m}

\newcommand{\NN}{\mathbb N}
\newcommand{\N}{\mathbb N}

\newtheorem{theorem}{Theorem}
\newtheorem{lemma}[theorem]{Lemma}
\newtheorem{claim}[theorem]{Claim}

\newtheorem{corollary}[theorem]{Corollary}
\theoremstyle{definition}
\newtheorem{definition}[theorem]{Definition}
\newtheorem{Theoremx}{Theorem}

\theoremstyle{remark}
\newtheorem{remark}[theorem]{Remark}
\newtheorem{example}[theorem]{Example}

\newtheorem{notation}[theorem]{Notation}

\numberwithin{theorem}{section}
\numberwithin{equation}{section}

\begin{document}

\title[Equimultiplicity theory of strongly $F$-regular rings]
{EQuimultiplicity theory of strongly $F$-regular rings}

\author{Thomas Polstra}
\thanks{Polstra was supported in part by NSF Postdoctoral Research Fellowship DMS $\#1703856$.}
\address{Department of Mathematics\\ University of Utah\\ Salt Lake City\\ UT 84112\\USA}
\email{polstra@math.utah.edu}

\author{Ilya Smirnov}
\address{Department of Mathematics\\
Stockholm University\\Stockholm,
SE - 106 91, Sweden
}
\email{smirnov@math.su.se}

\begin{abstract} We explore the equimultiplicity theory of the $F$-invariants Hilbert--Kunz multiplicity, $F$-signature, Frobenius Betti numbers, and Frobenius Euler characteristic over strongly $F$-regular rings. Techniques introduced in this article provide a unified approach to the study of localization of these invariants and detection of singularities.
\end{abstract}

\maketitle

\section{Introduction}

The most intrinsic feature of a ring $R$ of prime characteristic $p>0$ is the Frobenius endomorphism given by taking the $p$-powers, $x \mapsto x^p$. Let $F^e_*R$ be the $R$-module obtained by restricting scalars along the $e$th Frobenius endomorphism. For sake of simplicity assume that $(R,\mf m,k)$ is local and $F$-finite, meaning $R$ is a local ring and $F^e_*R$ is a finitely generated $R$-module for each $e\in \mathbb{N}$.
At the root of prime characteristic commutative algebra and algebraic geometry is Kunz's fundamental result characterizing flatness of the Frobenius endomorphism.

\begin{theorem}[\cite{Kunz1}]\label{theorem Kunz}
Let $(R, \mf m,k)$ be an $F$-finite local ring of prime characteristic $p > 0$. Then $R$ is regular if and only if $F^e_* R$ is a free $R$-module for some (equivalently, all) $e\in \N$.
\end{theorem}
Motivated by Kunz's theorem, it is natural to study non-regular prime characteristic rings by studying algebraic, geometric, and homological properties of the family of $R$-modules $\{F^e_*R\}_{e\in\N}$ which distinguish $R$ from a regular local ring.  We consider the following measurements:
\begin{enumerate}
\item $\mu(F^e_*R)$, the minimal number of generators of $F^e_*R$ as an $R$-module;
\item $a_e(R)$, the largest rank of a free summand of $F^e_*R$;
\item $\beta_i^e(R):=\dim_k(\mbox{Tor}_i^R(k,F^e_*R))$, the $i$th Betti number of $F^e_*R$;
\item $\chi_i^e(R):=\sum_{j=0}^i(-1)^j\beta^e_{i-j}(R)$.
\end{enumerate}
The asymptotic ratio of the above numbers, as compared with the rank of $F^e_*R$, produces several interesting and important numerical invariants unique to rings of prime characteristic.
\begin{enumerate}
\item Hilbert--Kunz multiplicity: $\displaystyle \ehk(R)=\lim_{e\to \infty}\mu(F^e_*R)/\rank(F^e_*R)$, \cite{Monsky}. 
\item $F$-signature: $\displaystyle \fsig(R)=\lim_{e\to \infty}a_e(R)/\rank(F^e_*R)$, \cite{HunekeLeuschke, SmithVanDenBergh, Tucker}.
\item The $i$th Frobenius Betti number: $\displaystyle \beta_i^F(R)=\lim_{e\to\infty}\beta_i^e(R)/\rank(F^e_*R)$, \cite{AberbachLi}.
\item The $i$th Frobenius Euler characteristic: $\displaystyle \chi_i^F(R)=\lim_{e\to\infty}\chi_i^e(R)/\rank(F^e_*R)$, \cite{DSPY3}.
\end{enumerate}

This article concerns the equimultiplicity theory of the above numerical invariants, a topic initiated by the second author in \cite{equi}. Specifically, we are interested in understanding when the above measurements are unchanged under localization. Our main result in this direction is the following:

\begin{Theoremx}\label{Main equimultiplicity theorem} Let $(R,\mf m)$ be an $F$-finite and strongly $F$-regular local ring of prime characteristic $p>0$ and let $P\in \Spec R$. Let $\{\nu_e(R)\}_{e\in \NN}$ be one of the following sequences of numbers: 
\begin{itemize}
    \item $\{\mu(F^e_*R)\}_{e\in \NN}$
    \item $\{a_e(R)\}_{e\in \NN}$
    \item $\{\beta_i^e(R)\}_{e\in \NN}$
    \item $\{\chi_i^e(R)\}_{e\in \NN}$
\end{itemize}  
Let $\displaystyle \nu(R)=\lim_{e\to \infty}\nu_e(R)/\rank(F^e_*R)$.
Then the following are equivalent:
\begin{enumerate}
\item $\nu(R)=\nu(R_P)$;
\item For each $e\in \mathbb{N}$, $\nu_e(R)=\nu_e(R_P)$.
\end{enumerate}
\end{Theoremx}

In the scenario that $\{\nu_e(R)\}_{e\in \NN}$ is the sequence of numbers $\{\mu(F^e_*R)\}_{e\in \NN}$ then Theorem~\ref{Main equimultiplicity theorem} is a significant improvement of \cite[Corollary~5.18]{equi}, where the same theorem was proven under the additional assumption that $R/P$ is a regular local ring.

It has been known for some time that Hilbert--Kunz multiplicity, $F$-signature, and Frobenius Betti numbers serve as measurements of singularities, see \cite{WatanabeYoshida, HunekeYao}, \cite{HunekeLeuschke,AberbachLeuschke}, and \cite{AberbachLi} respectively.  Frobenius Euler characteristic was developed in \cite{DSPY3} as a tool to prove that the functions $\beta_i^F\colon \Spec R\to \mathbb{R}$ sending $P\mapsto \beta_i^F(R_P)$ are upper semi-continuous and it was unclear from those techniques whether or not Frobenius Euler characteristic could be used to detect regular rings. Prior to this article, only the first Frobenius Euler characteristic was proven to serve as a measurement of singularity, see \cite[Main Theorem (iv)]{Li}. In the present article, we show Frobenius Euler characteristic does indeed serve as a measurement of singularities under the strongly $F$-regular hypothesis.

\begin{Theoremx}\label{Main Theorem equimultiplicity Frobenius Euler characteristic and Beti numbers} Let $(R,\fm)$ be an $F$-finite and strongly $F$-regular local ring of prime characteristic $p>0$. Then the following are equivalent:
\begin{enumerate}
\item $R$ is a regular local ring;
\item $\chi_i^e(R)=(-1)^i\rank(F^e_*R)$ for every $e\in \N$;
\item $\chi_i^e(R)=(-1)^i\rank(F^e_*R)$ for some $e\in \N_{\geq 1}$;
\item $\chi_i^F(R)=(-1)^i$.
\end{enumerate}
\end{Theoremx}

We conjecture that Theorem~\ref{Main Theorem equimultiplicity Frobenius Euler characteristic and Beti numbers} can be proven under weaker hypotheses. It seems likely that, similar to Hilbert--Kunz multiplicity and Frobenius Betti numbers, one would only need to assume the completion of the local ring $R$ at the maximal ideal has no low-dimensional components in order to know that $\chi_i^F(R)=(-1)^i$ implies $R$ is regular.  


The $F$-signature of a local ring $(R,\mf m, k)$ can be studied through splitting ideals, a notion originating in \cite{AberbachEnescu2}. For each $e\in \N$ the $e$th splitting ideal is 
\[
I_e=\{r\in R\mid \varphi(F^e_*r)\in \mf m, \forall \varphi\in \Hom_R(F^e_*R,R)\}
\]
and the $F$-signature of $R$ is realized as the limit $\fsig(R)=\lim_{e\to \infty}\length(R/I_e)/p^{e\dim(R)}$. To better understand the behavior of $F$-signature under localizations we consider relative splitting ideals: for each ideal $I\subseteq R$ and $e\in \N$ let 
\[
I_e(I)=\{r\in R\mid \varphi(F^e_*r)\in I, \forall \varphi\in \Hom_R(F^e_*R,R)\}.
\]
Observe that $I_e(\mf m)=I_e$. If $I\subseteq R$ is an $\mf m$-primary ideal then we can define an $F$-signature relative to the ideal $I$ as $\fsig(I)=\lim_{e\to \infty}\length(R/I_e(I))/p^{e\dim(R)}$, a limit we will observe exists.  Not only do relative splitting ideals allow us to understand the behavior of $F$-signature in the context of Theorem~\ref{Main Theorem equimultiplicity Frobenius Euler characteristic and Beti numbers}, we prove the following associativity type formula for $F$-signature which is of independent interest.
The corresponding formula for Hilbert--Samuel multiplicity is 
a classic and very useful result of Lech (\cite{Lech}),
the version for Hilbert--Kunz multiplicity can be found in \cite[Proposition~5.4]{equi} and was extensively used therein.

\begin{Theoremx}\label{Main Theorem Associativity type formula} 
Let $(R, \mf m)$ be an $F$-finite local domain of prime characteristic $p>0$. Suppose that $I\subseteq R$ is an ideal such that $R/I$ is Cohen-Macaulay of dimension $h$ and $\underline{x}=x_1,\ldots, x_h$ a parameter sequence on $R/I$. Then 
\[
\lim_{n_1,\ldots,n_h\to \infty}\frac{1}{n_1\cdots n_h}\fsig(I,(x_1^{n_1},\ldots, x_h^{n_h}))=\sum_{P} \eh(x_1,\ldots,x_h;R/P)\fsig(IR_P),
\]
where the sum is taken over all prime ideals $P\supseteq I$ such that $\dim(R/I)=\dim(R/P)$.
\end{Theoremx}

Section~\ref{Section preliminary} contains background results and basic properties of splitting ideals relative to an ideal. The proofs of Theorem~\ref{Main equimultiplicity theorem} and Theorem~\ref{Main Theorem equimultiplicity Frobenius Euler characteristic and Beti numbers} can be found in Section~\ref{Section Equimultiplicity}. We also use Section~\ref{Section Equimultiplicity} to further explore the behavior of splitting ideals. For example, see Theorem~\ref{Theorem depth R/P vs depth R/I_e(P)} for a proof that $\depth(R/P)=\depth(R/I_e(P))$ whenever $P$ is a prime ideal of a strongly $F$-regular local ring satisfying $\fsig(R)=\fsig(R_P)$. Section~\ref{Section Associativity} is devoted to proving Theorem~\ref{Main Theorem Associativity type formula}.

\subsection*{Acknowledgements} The authors thank Alessandro De Stefani for valuable feedback on a preliminary draft of this article. 

\section{Preliminary Results}\label{Section preliminary}

\subsection{Hilbert--Kunz multiplicity}

Monsky's introduction of Hilbert--Kunz multiplicity is a continuation of Kunz's work on prime characteristic rings in \cite{Kunz1, Kunz2}.

\begin{definition}\label{HK definition}
Let $(R, \mf m)$ be a local ring of prime characteristic $p>0$ and $I$ be 
an $\mf m$-primary ideal. 
Denote $\frq{I} = (x^{p^e} \mid x \in I)$. Then
the Hilbert--Kunz multiplicity of $I$ is 
\[
\ehk(I) = \lim_{e \to \infty} \frac{\length (R/\frq{I})}{p^{e\dim R}}.
\]
\end{definition}
The Hilbert--Kunz multiplicity of a local ring $(R,\mf m,k)$ is the Hilbert--Kunz multiplicity of the maximal ideal $\mf m$ and is denoted by $\ehk(R)$. If $R$ is an $F$-finite domain then $\rank( F^e_*R) = p^{e \dim R}[k:k^{p^e}]$ by \cite[Proposition~2.3]{Kunz2} and therefore 
\[
\ehk(R)=\lim_{e\to \infty}\frac{\lambda(R/\mf m^{[p^e]})}{p^{e\dim(R)}}=\lim_{e\to\infty}\frac{\mu(F^e_*R)}{\rank(F^e_*R)}.
\]
Hence the definition of Hilbert--Kunz multiplicity presented in the introduction agrees with Definition~\ref{HK definition}.

\subsection{$F$-signature and splitting ideals}
Below is a definition, due to Tucker, which is a natural generalization of the splitting ideals and presents a natural extension of $F$-signature.

\begin{definition}\label{degen def}
Let $R$ be an $F$-finite ring and $\fa$ be an ideal.  
The $e$th splitting ideal of $\fa$ is defined as 
\[
I_e(\fa) = \{r\in R\mid \varphi(F^e_*r)\in \mf a, \forall \varphi\in\Hom_R(F^e_*R,R)\}.
\] 
\end{definition}
We record the following basic properties concerning splitting ideals, many of which mimic the behavior of the standard splitting ideals $I_e=I_e(\mf m)$.
\begin{lemma}\label{Basic properties}
Suppose $(R,\fm)$ is an $F$-finite local ring of prime characteristic $p>0$ and Krull dimension $d$. Let 
$\fa \subset R$ be an ideal.
Then the sequence of ideals $\{I_e(\fa)\}$ satisfies the following properties:
\begin{enumerate}
\item $I_e(\fa)$ is an ideal;
\item $\fa^{[p^e]}\subseteq I_e(\fa)$;
\item $I_e(\fa)^{[p]}\subseteq I_{e+1}(\fa)$;
\item $\varphi(F^{e_0}_*I_{e+e_0}(\fa))\subseteq I_e(\fa)$ for every $e,e_0\in\mathbb{N}$ and $\varphi\in\Hom_R(F^{e_0}_*R,R)$;
\item If $\fa$ is $\fm$-primary then the limit $\fsig(\mf a):=\lim\limits_{e\to \infty }\frac{\length(R/I_e(\fa))}{p^{ed}}$ exists and $\length(R/I_e(\fa))=\fsig(\fa)p^{ed}+O(p^{e(d-1)})$. The value $\fsig(\fa)$ is referred to as the $F$-signature of $\fa$;
\item If $W$ is a multiplicative set then $I_e(\fa)R_W=I_e(\fa R_W)$;
\item $I_e(\fa : J) = I_e(\fa) : J^{[p^e]}$ for all ideals $J$;
\item If $P$ is a prime ideal then $I_e(P)$ is $P$-primary;
\item If $x\in R$ is regular element of $R/\mf a$ then $x$ is regular element of $R/I_e(\mf a)$ for every $e\in \NN$; 
\item If $R$ is a regular local ring then $I_e(\fa)=\fa^{[p^e]}$ for every $e\in \N$;
\item If $\mf b\subseteq R$ is an ideal and $\mf a\subseteq \mf b$ then $I_e(\mf a)\subseteq I_e(\mf b)$;
\end{enumerate}
\end{lemma}

\begin{proof} The proofs of (1)-(4) are straightforward and are left to the reader. Statement (5) then follows by  \cite[Corollary~4.5]{PolstraTucker}. To prove (6) it is enough to observe $\Hom_R(F^e_*R,R)_W\cong \Hom_{R_W}(F^e_*R_W, R_W)$. For (7) we note that $a \in I_e(\fa : J)$ if and only if  for all $\varphi\in \Hom_R(F^e_*R,R)$ we have $\varphi(F^e_*a) \in (\fa:J)$, or equivalently, $\varphi(F^e_*J^{[p^e]} a) = J \varphi(F^e_*a) \subseteq \fa$.  Statements (8) and (9) easily follow from (7). Observation (10) follows from Theorem~\ref{theorem Kunz}; if $F^e_*R$ is free then it is then easy to see that $F^e_*I_e(\fa)=\fa F^e_*R$ from which it follows that $I_e(\fa)=\fa^{[p^e]}$. Property (11) is trivial.
\end{proof}

\begin{corollary}\label{cor pos diff}
Let $(R,\mf m)$ be an $F$-finite and $F$-pure local ring of prime characteristic $p>0$. If $J\subsetneq I$ are ideals then $I_e(J)\subsetneq I_e(I)$. Moreover, if $R$ is strongly $F$-regular and $J, I$ are  $\mf m$-primary, then $\fsig(J) > \fsig(I)$.  
\end{corollary}
\begin{proof}
Without loss of generality we may assume $I = (J, x)$ and $x\not\in J$.
Then $I_e(J):x^{p^e} = I_e(J:x) \subseteq I_e(\mf m) \neq R$, so 
$I_e(J) \subsetneq I_e(J) + (x^{p^e}) \subseteq I_e(J, x).$ 

For the second part, observe that 
\[
\length \left(\frac{I_e(J, x)}{I_e(J)} \right)
\geq \length \left(\frac{I_e(J) + (x^{p^e})}{I_e(J)} \right)
= \length (R/(I_e(J):x^{p^e})) \geq \length (R/I_e(\mf m)).
\]
Therefore $\fsig(J)-\fsig((J,x))\geq \fsig(R) > 0$.
\end{proof}

Similar to the usual $F$-signature, the $F$-signature of an $\mf m$-primary ideal $\mf a$ is realized as the limit of normalized Hilbert--Kunz multiplicities of the ideals $I_e(\mf a)$.

\begin{theorem} Let $(R,\mf m)$ be an $F$-finite and reduced local ring of prime characteristic $p>0$ and $\mf a$ an $\mf m$-primary ideal. Then
\[
\fsig(\mf a)=\lim_{e\to \infty}\frac{\ehk(I_e(\mf a))}{p^{e\dim(R)}}.
\]
\end{theorem}

\begin{proof}
The assertion is equivalent to saying that
\[
\lim_{e \to \infty} \frac{1}{p^{e\dim R}}
\left|\length (R/I_e(\fa)) - \ehk(I_e(\fa))\right|= 0,
\]
which is the content of \cite[Corollary~3.7]{Tucker}.
\end{proof}

\section{Equimultiplicity of $F$-invariants}\label{Section Equimultiplicity}

We are interested in understanding when the invariants Hilbert--Kunz multiplicity, $F$-signature, Frobenius Betti numbers, and Frobenius Euler characteristics are unchanged under localization. Work of the second author in \cite{equi} began this study for Hilbert--Kunz multiplicity where the following was proven:

\begin{theorem}[{\cite[Corollary~5.16]{equi}}]\label{SmirnovEquiTheorem}
Let $(R,\fm)$ be an excellent weakly $F$-regular local ring of prime characteristic $p>0$ and $P\subset R$ a prime ideal such that $R/P$ is a regular local ring. Then the following are equivalent:
\begin{enumerate}
\item $\ehk(R)=\ehk(R_P)$,
\item for each $e\in \N$, $\length(R/\fm^{[p^e]})/p^{e\dim(R)}=\length(R_P/P^{[p^e]}R_P)/p^{e\height(P)}$.
\end{enumerate}
\end{theorem}

The techniques surrounding Theorem~\ref{SmirnovEquiTheorem} involve a careful and challenging analysis of the behavior of the ideals $\{P^{[p^e]}\}$ and $\{(P,\underline{x})^{[p^e]}\}$ where $\underline{x}$ is a regular system of parameters modulo $P$. Using elementary techniques, we recover the above theorem without assuming $R/P$ is a regular local ring, but we do replace the assumption of weakly $F$-regular with the conjecturally equivalent assumption that $R$ is strongly $F$-regular. Our techniques stem from a novel, yet simple, observation that if a module $M$ is a direct summand of $F^{e_0}_*R$ for some $e_0\in\mathbb{N}$ and $R$ is strongly $F$-regular, then asymptotically there will be many direct summands of $F^e_*R$ isomorphic to $M$ as $e\to \infty$. To make this precise, we begin with some notation.

\begin{notation}
Let $R$ be a ring and $N \subseteq M$ be finitely generated $R$-modules. 
Let $\srk_N (M)$ denote the maximal number of $N$-summands appearing in all possible direct sum decompositions of $M$.
\end{notation}

The following lemma should be compared with \cite[Proposition~3.3.1]{SmithVanDenBergh}.

\begin{lemma}\label{summand lemma}
Let $(R,\mathfrak{m})$ be an $F$-finite and strongly $F$-regular local ring. Suppose $M$ is a finitely generated $R$-module such that $\srk_M (F_*^{e_0}R) > 0$ for some $e_0\in \mathbb{N}$. Then 
\[
\liminf_{e\to \infty}\frac{\srk_M (F_*^e R)}{\rank(F^e_*R)}>0.
\]
\end{lemma}

\begin{proof}
 Suppose that $F^{e_0}_*R\cong M\oplus N$. For each $e\in \mathbb{N}$ write $F^e_*R\cong R^{\oplus a_e(R)}\oplus M_e$. Then $F^{e+e_0}_*R\cong F^{e_0}_*R^{\oplus a_e(R)}\oplus F^{e_0}_*M_e$ and it follows that $M^{\oplus a_e(R)}$ is a direct summand of $F^{e+e_0}_*R$. In particular, $\srk_M (F_*^{e+e_0} R)\geq a_e(R)$ and therefore
 \[
 \liminf_{e\to \infty}\frac{\srk_M (F_*^eR)}{\rank(F^e_*R)}\geq \liminf_{e\to \infty}\frac{a_{e-e_0}(R)}{\rank(F^e_*R)}=\frac{\fsig(R)}{\rank(F^{e_0}_*R)}>0.
 \]
\end{proof}

\subsection{$F$-signature and splitting ideals}

We are prepared to present a proof of Theorem~\ref{Main equimultiplicity theorem} for $F$-signature. But first:

\begin{remark}\label{remark free summands}
To make full use of Lemma~\ref{summand lemma} in the following theorem we remind the reader that the maximal rank of a free summand of a finitely generated module $M$ over a local ring $R$ is invariant of a choice of a direct sum decomposition.  This is because $R$ is a direct summand of $M$ if and only if $\hat{R}$ is a direct summand of $\hat{M}$ and a complete local ring satisfies the Krull--Schmidt condition.
\end{remark}

\begin{theorem}\label{theorem signature the same}
Let $(R,\mathfrak{m})$ be a strongly $F$-regular and $F$-finite local ring. Suppose $P\subset R$ is a prime ideal. Then $\fsig(R)=\fsig(R_P)$ if and only if $a_e(R)=a_e(R_P)$ for every $e\in \mathbb{N}$.
\end{theorem}

\begin{proof}
If $a_e(R)=a_e(R_P)$ for every $e\in \NN$ then it is trivial that $\fsig(R)=\fsig(R_P)$ since $\rank_R(F^e_*R)=\rank_{R_P}(F^e_*R_P)$.

 Suppose that $a_{e_0}(R_P)>a_{e_0}(R)$ and write $F^{e_0}_*R\cong R^{\oplus a_{e_0}}\oplus M_{e_0}$. Then $(M_{e_0})_P$ has a free $R_P$-summand. For each $e\in \mathbb{N}$ by Remark~\ref{remark free summands} we may write 
 \[F^e_*R\cong R^{\oplus a_e(R)}\oplus M_{e_0}^{\oplus \srk_{M_{e_0}}(F_*^e R)}\oplus N_e.\] 
 Localizing at the prime $P$ we see that $a_e(R_P)\geq a_e(R)+\srk_{M_{e_0}}(F_*^e R)$ and
\begin{eqnarray*}
\displaystyle \fsig(R_P)=\lim_{e\to \infty}\frac{a_e(R_P)}{\rank(F^e_*R)}& \displaystyle \geq \lim_{e\to \infty}\frac{a_e(R)}{\rank(F^e_*R)}+\liminf_{e\to \infty}\frac{\srk_{M_{e_0}}(F_*^e R)}{\rank(F^e_*R)}\\
 &\displaystyle =\fsig(R)+\liminf_{e\to \infty}\frac{\srk_{M_{e_0}}(F_*^e R)}{\rank(F^e_*R)}.
\end{eqnarray*} 
 Therefore $\fsig(R_P)>\fsig(R)$ by Lemma~\ref{summand lemma}.
\end{proof}

The following theorem states that the splitting ideals of $R$ and that of a localization $R_P$ can be effectively compared whenever the Frobenius splitting numbers of $R$ and $R_P$ agree.

\begin{theorem}\label{theorem splitting numbers same}
Let $(R,\fm)$ be an $F$-finite local ring of prime characteristic $p>0$, $P$ be a prime ideal.
Then the following are equivalent:
\begin{enumerate}
    \item $a_e(R)=a_e(R_P)$,
    \item $I_e((P, I))=I_e(P)+I^{[p^e]}$ for all ideals $I$,
    \item $I_e(\fm)=I_e(P)+\fm^{[p^e]}$.
\end{enumerate}
\end{theorem}
\begin{proof}
Write $F^e_*R\cong R^{\oplus a_e(R)}\oplus M_e$. By definition, $F^e_*I_e(P)=P^{\oplus a_e(R)}\oplus \{\eta\in M_e\mid \varphi(\eta)\in P,\,\forall\varphi\in \Hom_R(M_e,R)\}$. Hence, $a_e(R)=a_e(R_P)$ if and only if $\Hom_R (M_e, R) = \Hom_R (M_e, P)$. 
It follows then that $\Hom_R (M_e, R) = \Hom_R (M_e, P + I)$, so 
$F^e_*(I_e(P)+I^{[p^e]})=(P,I)^{\oplus a_e(R)}\oplus (M_e+IM_e)=(P,I)^{\oplus a_e(R)}\oplus M_e=F^e_*I_e(P + I)$. Thus (1) implies (2).

Since (2) trivially implies (3) it is left to show that the last condition implies the first.
Suppose that $I_e(\fm)=(I_e(P),\fm^{[p^e]})$. Then
\[
F_*I_e(\mf m) = 
\fm^{\oplus a_e(R)}\oplus M_e=\fm^{\oplus a_e(R)}\oplus \left( \{\eta\in M_e\mid \varphi(\eta)\in P,\,\forall\varphi\in \Hom_R(M_e,R)\} + \fm M_e\right).
\]
By Nakayama's lemma we then get that 
\[
M_e= \{\eta\in M_e\mid \varphi(\eta)\in P,\,\forall\varphi\in \Hom_R(M_e,R)\},
\]
i.e., $\Hom_R(M_e,R)=\Hom_R(M_e,P)$ and therefore $a_e(R)=a_e(R_P)$.
\end{proof}

Theorem~\ref{theorem signature the same} and Theorem~\ref{theorem splitting numbers same} imply the following:

\begin{corollary}\label{Corollary $F$-sig same ideals} Let $(R,\fm)$ be an $F$-finite and strongly $F$-regular local ring of prime characteristic $p>0$ and $P$ be a prime ideal. Then $\fsig(R)=\fsig(R_P)$ if and only if 
$ I_e(\fm)=I_e(P)+\fm^{[p^e]}$ for every $e\in \NN$.
\end{corollary}

The techniques surrounding Theorem~\ref{theorem signature the same} provide a novel proof that the $F$-signature of a local ring is $1$ if and only if $R$ is a regular local ring. 

\begin{theorem}[{\cite[Corollary~16]{HunekeLeuschke}}]\label{theorem signature 1 iff regular} Let $(R,\fm)$ be an $F$-finite local ring of prime characteristic $p>0$. Then $\fsig(R)=1$ if and only if $R$ is a regular local ring.
\end{theorem}

\begin{proof}
Having positive $F$-signature implies $R$ is strongly $F$-regular.\footnote{The converse also holds, see \cite[Main Theorem]{AberbachLeuschke}.} Hence, $R$ is a domain, so $R_{(0)}$ is a regular ring 
and, therefore, $\fsig(R_{(0)}) = \fsig(R)$.
Invoking Theorem~\ref{theorem signature the same} 
we then have that $a_e(R)=a_e(R_{(0)})=\rank(F^e_*R)$. Therefore $F^e_*R$ is a free $R$-module and $R$ is a regular local ring by Theorem~\ref{theorem Kunz}.
\end{proof}

The advantage of the proof of Theorem~\ref{theorem signature 1 iff regular} 
is that it directly uses Kunz's Theorem while the proof of \cite[Corollary~16]{HunekeLeuschke} invokes the fact that $R$ must be regular if $\ehk(R) = 1$ (\cite{WatanabeYoshida, HunekeYao}). We may also adapt our approach to give a somewhat novel proof that Hilbert--Kunz multiplicity of a formally unmixed local ring is $1$ if and only if $R$ is a regular local ring, see Theorem~\ref{Theorem Hilbert--Kunz the same} below.


\begin{theorem}
\label{Theorem depth R/P vs depth R/I_e(P)}
Let $(R,\mf m)$ be an $F$-finite and strongly $F$-regular local ring of prime characteristic $p>0$. Suppose that $P\in \Spec R$, $\fsig(R)=\fsig(R_P)$, and $\underline{x}=x_1,\ldots,x_h$ is a sequence of elements in $R$. Then the following are equivalent:
\begin{enumerate}
\item $\underline{x}$ is a regular sequence on $R/P$;
\item $\underline{x}$ is a regular sequence on $R/I_e(P)$ for each $e\in \N$;
\item $\underline{x}$ is a regular sequence on $R/I_e(P)$ for some $e\in \NN$.
\end{enumerate}
In particular, $\depth (R/P)=\depth (R/I_e(P))$ for all $e\in \mathbb{N}$.
\end{theorem}

\begin{proof}
Let $x_1,\ldots,x_h$ be a regular sequence on $R/P$. To show that $x_1,\ldots, x_h$ is a regular sequence on $R/I_e(P)$ it is equivalent to check that for every $0\leq i\leq h-1$
\[
(I_e(P),x^{p^e}_1,\ldots,x_i^{p^e}):x_{i+1}^{p^e}=(I_e(P),x^{p^e}_1,\ldots,x_i^{p^e}).
\]
By Theorem~\ref{theorem signature the same} and  Theorem~\ref{theorem splitting numbers same} we have that
$
(I_e(P),x^{p^e}_1,\ldots,x_i^{p^e})=I_e(P,x_1,\ldots, x_i)
$
and by (7) of Lemma~\ref{Basic properties} we have that
$
(I_e(P,x_1,\ldots, x_i)):x_{i+1}^{p^e}=I_e((P,x_1,\ldots, x_i):x_{i+1}).
$
But $x_1,\ldots, x_h$ is a regular sequence on $R/P$ and therefore by a second application of Theorem~\ref{theorem signature the same} and Theorem~\ref{theorem splitting numbers same} we see that
\[
I_e((P,x_1,\ldots, x_i):x_{i+1})=I_e(P,x_1,\ldots, x_i)=(I_e(P),x^{p^e}_1,\ldots, x^{p^e}_i).
\]

Now suppose that for some $e\in \mathbb{N}$ that $x_1,\ldots, x_h$ is a regular sequence on $R/I_e(P)$. Then for each $0\leq i\leq h-1$ we have by (7) of Lemma~\ref{Basic properties}, Theorem~\ref{theorem signature the same}, and Theorem~\ref{theorem splitting numbers same} that 
\begin{eqnarray*}
I_e((P,x_1,\ldots,x_i):x_{i+1})&=&(I_e(P,x_1,\ldots,x_i):x^{p^e}_{i+1})=(I_e(P),x_1^{p^e}\ldots, x_{i-1}^{p^e}):x_{i}^{p^e}\\
&=&(I_e(P),x_1^{p^e}\ldots, x_{i-1}^{p^e})=I_e((P,x_1,\ldots,x_{i-1})).
\end{eqnarray*}
By Corollary~\ref{cor pos diff} we must have that
$
(P,x_1,\ldots,x_i):x_{i+1}=(P,x_1,\ldots,x_i)
$
for each $0\leq i\leq h-1$ and therefore $x_1,\ldots,x_h$ is indeed a regular sequence on $R/P$.
\end{proof}

Let $(R,\mf m)$ be an $F$-finite and strongly $F$-regular local ring of prime characteristic $p>0$. Observe by (9) of Lemma~\ref{Basic properties} that $\depth (R/I_e(P))\geq 1 $ for every $e\in \NN$ and $P\in \Spec R\setminus \{\mf m\}$.  However, it does not follow that $\depth (R/I_e(P))=\depth(R/P)$ if we do not assume $\fsig(R)=\fsig(R_P)$.

\begin{example}
Consider the regular local ring $S$ of prime characteristic $2$ obtained by localizing $\mathbb{F}_2[x,y,z,w]$ at the maximal ideal $(x,y,z,w)$ and let $R=S/(xy-zw)$. Then $R$ is a strongly $F$-regular isolated singularity. Consider the height $1$ prime ideal $P=(x,z)$. By the techniques surrounding Fedder's criterion \cite{Fedder}, c.f. \cite[Theorem~2.3]{Glassbrenner}, for each $e\in \N$ we have that
\[
I_e(R)=\frac{P^{[2^e]}:_S(xy-zw)^{2^e-1}}{(xy-zw)}.
\]
Observe that $R/P$ is a regular local ring of dimension $2$, yet one can check that $I_1(R)=( xz, x^2,z^2)$ and $R/I_1(R)$ has depth $1$.
\end{example}

\subsection{Hilbert-Kunz multiplicity} Now, we prove Theorem~\ref{Main equimultiplicity theorem} for Hilbert--Kunz multiplicity.

\begin{theorem}\label{Theorem Hilbert--Kunz the same}
Let $(R,\fm)$ be a strongly $F$-regular and $F$-finite local ring of dimension $d$ and $P\in \Spec R$. Then the following are equivalent:
\begin{enumerate}
    \item $\ehk(R)=\ehk(R_P)$;
    \item $\length(R/\fm^{[p^e]})/p^{ed}=\length(R_P/P^{[p^e]})/p^{e\height(P)}$ for every $e\in \NN$;
    \item $\mu(F^e_*R)=\mu(F^e_*R_P)$ for every $e\in \NN$;
    \item $F^e_*R/PF^e_*R$ is a free $R/P$-module for every $e\in \NN$.
\end{enumerate}  
\end{theorem}

\begin{proof} Conditions (2) and (3) are equivalent by \cite[Proposition~2.3]{Kunz2} and clearly conditions (2) and (3) imply (1). To show that condition (1) implies condition (3)  suppose that $\mu(F^{e_0}_*R)>\mu(F^{e_0}_*R_P)$. If we write $F^{e_0}_* R\cong R^{\oplus a_{e_0}(R)}\oplus M_{e_0}$ then $\mu(M_{e_0})>\mu((M_{e_0})_P)$. Let $b_e = \srk_{M_{e_0}} (F^e_*R)$. By Remark~\ref{remark free summands} we may write $F^e_*R\cong R^{\oplus a_e} \oplus (M_{e_0})^{\oplus b_e}\oplus N_e$ and it follows that
\[
\ehk(R)=\lim_{e\to \infty}\frac{1}{\rank(F^e_*R)}(a_e(R)+b_e\mu(M_{e_0})+\mu(N_e))
\]
and
\[
\ehk(R_P)=\lim_{e\to \infty}\frac{1}{\rank(F^e_*R)}(a_e(R)+b_e\mu((M_{e_0})_P)+\mu((N_e)_P)).
\]
Therefore
\begin{align*}
\ehk(R_P)&\leq \lim_{e\to \infty}\frac{1}{\rank(F^e_*R)}\left(a_e(R)+b_e(\mu(M_{e_0})-1)+\mu(N_e)\right)\\
&\leq \lim_{e\to \infty}\frac{1}{\rank(F^e_*R)}\left (a_e(R)+b_e\mu(M_{e_0})+\mu(N_e) \right )-\liminf_{e\to \infty}\frac{b_e}{\rank(F^e_*R)}\\
&=\ehk(R)-\liminf_{e\to \infty}\frac{b_e}{\rank(F^e_*R)},
\end{align*}
a value strictly less than $\ehk(R)$ by Lemma~\ref{summand lemma}.

Now suppose that $\ehk(R)=\ehk(R_P)$. To show that $F^e_*R/PF^e_*R$ is a free $R/P$-module observe first that by Nakayama's Lemma, $\mu_{R_P}(F^e_*R_P)=\mu_{R_P}(F^e_*R_P/PF^e_*R_P)$. Therefore
\[
\mu_R(F^e_*R/PF^e_*R)= \mu_R(F^e_*R)=\mu_{R_P}(F^e_*R_P)=\mu_{R_P}(F^e_*R_P/PF^e_*R_P)\leq \mu_R(F^e_*R/PF^e_*R). 
\]
Therefore, as an $R/P$-module, we have that $F^e_*R/PF^e_*R$ is generated by $\rank_{R/P}(F^e_*R/PF^e_*R)$ elements and must be free.

Conversely, if $F^e_*R/PF^e_*R$ is a free $R/P$-module for every $e\in \NN$ then $\mu_R(F^e_*R/PF^e_*R)=\mu_{R_P}(F^e_*R_P/PF^e_*R_P)$ and therefore
\[
\mu_R(F^e_*R)=\mu_R(F^e_*R/PF^e_*R)=\mu_{R_P}(F^e_*R_P/PF^e_*R_P)=\mu_{R_P}(F^e_*R_P).
\]
\end{proof}

The following corollary is the analogue of Theorem~\ref{Theorem depth R/P vs depth R/I_e(P)} for Hilbert--Kunz multiplicity.

\begin{corollary}
\label{corollary to theorem on HK}
Let $(R,\mf m)$ be a strongly $F$-regular and $F$-finite local ring of prime characteristic $p>0$. Suppose that $P\in \Spec R$ and $\ehk(R)=\ehk(R_P)$.  Then for each sequence of elements $\underline{x}=x_1,\ldots,x_h$ the following are equivalent:
\begin{enumerate}
\item $\underline{x}$ is a regular sequence on $R/P$;
\item $\underline{x}$ is a regular sequence on $R/P^{[p^e]}$ for each $e\in \N$;
\item $\underline{x}$ is a regular sequence on $R/P^{[p^e]}$ for some $e\in \N$.
\end{enumerate}
In particular, $\depth(R/P)=\depth(R/P^{[p^e]})$ for every $e\in \N$.
\end{corollary}
\begin{proof}
For any finitely generated $R$-module $M$ a sequence of elements $\underline{x}$ is a regular sequence on $M$ if and only if $\underline{x}$ is a regular sequence on $F^e_*M$. The corollary is immediate by Theorem~\ref{Theorem Hilbert--Kunz the same} since the modules $F^e_*(R/P^{[p^e]})\cong F^e_*R/PF^e_*R$ are free $R/P$-modules. 
\end{proof}

Corollary~\ref{corollary to theorem on HK} is an improvement of an observation that can be made from \cite[Proposition~3.1 and Corollary~5.19]{equi}: if $(R,\mf m)$ is weakly $F$-regular, $P\in \Spec R$ satisfies $\ehk(R)=\ehk(R_P)$, and $R/P$ is regular then $R/P^{[p^e]}$ is Cohen-Macaulay for $e\in \N$.

We utilize Theorem~\ref{Theorem Hilbert--Kunz the same} and results of \cite{AberbachEnescuLowerBounds} and provide a novel proof that the Hilbert--Kunz multiplicity of a local ring is $1$ if and only if the ring is regular. We recall that a ring is unmixed if it is equidimensional and has no embedded components.

\begin{theorem}[\cite{WatanabeYoshida}] Let $(R,\fm)$ be a formally unmixed local $F$-finite ring of prime characteristic $p>0$. Then $\ehk(R)=1$ if and only if $R$ is a regular local ring.
\end{theorem}

\begin{proof}
The assumption on Hilbert--Kunz multiplicity implies that $R$ is strongly $F$-regular, see \cite[Corollary~3.6]{AberbachEnescuLowerBounds}. By Theorem~\ref{Theorem Hilbert--Kunz the same} applied to $P = (0)$, $\mu (F_* R) = \rank F_* R$, so $F_* R$ is a free $R$-module
and $R$ is regular by Theorem~\ref{theorem Kunz}.
\end{proof}

\subsection{Frobenius Betti numbers and Frobenius Euler characteristic} We now turn our attention to the behavior of Frobenius Betti numbers and Frobenius Euler characteristics under localizations.

\begin{definition}
Let $(R,\fm)$ be an $F$-finite local domain of prime characteristic $p>0$. For each $e\in \N$ let $\Omega_i^e(R)$ be the $i$th syzygy in the minimal free resolution of $F^e_*R$.  The $i$th Frobenius Betti number of $R$ is
\[
\beta_i^F(R)=\lim_{e\to \infty}\frac{\mu(\Omega^e_i(R))}{\rank(F^e_*R)}
\]
and the $i$th Frobenius Euler characteristic of $R$ is 
\[
\chi_i^F(R)=\lim_{e\to \infty}\sum_{j=0}^i(-1)^j\frac{\mu(\Omega_{i-j}^e(R))}{\rank(F^e_*R)}=\sum_{j=0}^i(-1)^j\beta_{i-j}^F(R).
\]
\end{definition}

We refer the reader to \cite{Li, AberbachLi, DSHNB, DSPY3} for basics on Frobenius Betti numbers and \cite{DSPY3} for basics on Frobenius Euler characteristic. Our study begins with a simple application of the Auslander--Buchsbaum formula.

\begin{lemma}\label{projective dimension lemma} Let $(R,\fm)$ be a local ring of prime characteristic $p>0$. The following are equivalent:
\begin{enumerate}
\item $R$ is a regular local ring;
\item $F^e_*R$ has finite projective dimension as an $R$-module for every $e\geq 1$;
\item $F^e_*R$ has finite projective dimension for some $e\in\N$.
\end{enumerate}
\end{lemma}

\begin{proof}
It is easy to see that $\depth(R)=\depth(F^e_*R)$ for every $e\in \N$. Hence by the Auslander--Buchsbaum formula, if the projective dimension of $F^e_*R$ is finite then $F^e_*R$ is a free $R$-module and the lemma follows from Theorem~\ref{theorem Kunz}.
\end{proof}

\begin{lemma}\label{rank of syzygy lemma} Let $(R,\fm)$ be an $F$-finite local domain of prime characteristic $p>0$. Then 
\[
\rank_R(\Omega_i^e(R))=\chi_{i-1}^e(R)+(-1)^{i}\rank(F^e_*R).
\]
Moreover, if $R$ is not regular then $\beta_i^e(R)> \rank_R(\Omega_i^e(R))=\chi_{i-1}^e(R)+(-1)^{i}\rank(F^e_*R)$.
\end{lemma}

\begin{proof} Rank is additive on exact sequences and there are long exact sequences
\[
0\to \Omega_i^e(R) \to R^{\oplus \beta_{i-1}^e(R)}\to \cdots \to R^{\oplus \mu(F^e_*R)}\to F^e_*R\to 0.
\]
By Lemma~\ref{projective dimension lemma}, if $R$ is not regular, then $\Omega_i^e(R)$ is not free, hence $\beta_i^e(R)=\mu(\Omega_i^e(R))>\rank(\Omega_i^e(R))$.
\end{proof}

\begin{lemma}\label{lower bound for Euler characteristic} Let $(R,\fm)$ be a local $F$-finite domain of prime characteristic $p$ and let $e,i\in \N$ with $e\geq 1$. Then $\chi^e_i(R)\geq (-1)^i \rank(F^e_*R)$ with equality if and only if $R$ is a regular local ring.
\end{lemma}

\begin{proof} For $i = 0$ the lemma follows from Theorem~\ref{theorem Kunz}. If $i\geq 1$ and $e\in \N$ then $\chi_i^e(R)=\beta_i^e(R)-\chi_{i-1}^e(R)$. Applying Lemma~\ref{rank of syzygy lemma} we arrive at 
\[
\chi_i^e(R)\geq \chi_{i-1}^e(R)+(-1)^i\rank(F^e_*R)- \chi_{i-1}^e(R)=(-1)^i\rank(F^e_*R)
\]
with equality if and only if $R$ is a regular local ring.
\end{proof}

\begin{lemma}\label{Equality of Frobenius Euler characteristic} Let $(R,\fm)$ be a local $F$-finite domain and $P\in \Spec R$. Then $\beta_i^e(R)=\beta_i^e(R_P)$ if and only if $\chi_i^e(R)=\chi_i^e(R_P)$ and $\chi_{i-1}^e(R)=\chi_{i-1}^e(R_P)$. In particular, if $\beta_1^e(R)=\beta_1^e(R_P)$ then $\mu(F^e_*R)=\mu(F^e_*R_P)$.
\end{lemma}

\begin{proof}
Observe first that
\[
\beta_i^e(R)=\chi_i^e(R)+\chi_{i-1}^e(R)
\]
and 
\[
\beta_i^e(R_P)=\chi_i^e(R_P)+\chi_{i-1}^e(R_P).
\]
Suppose that $\beta_i^e(R)=\beta_i^e(R_P)$. The function $\chi_{i}^e\colon\Spec R\to \mathbb{R}$ is upper semicontinuous, \cite[Proposition~3.1]{DSPY3}, therefore
$
\chi_i^e(R)\geq \chi_i^e(R_P).
$
If $\beta_i^e(R)=\beta_i^e(R_P)$ then 
\[
\chi_{i-1}^e(R_P)\geq \chi_{i-1}^e(R),
\]
but the function $\chi_{i-1}^e\colon\Spec R\to \mathbb{R}$ is also upper semicontinuous, therefore equality must hold.
\end{proof}

Similar to Lemma~\ref{summand lemma}, if $R$ is strongly $F$-regular and a module $M$ appears as a direct summand of $\Omega_i^{e_0}(R)$ for some $e_0\in \mathbb{N}$ then $M$ appears as a direct summand of $\Omega_i^e(R)$ asymptotically many times as $e\to \infty$.

\begin{lemma}\label{summands of syzygies lemma} Let $(R,\fm)$ be an $F$-finite and strongly $F$-regular local ring of prime characteristic $p>0$ and $M$ be a finitely generated $R$-module. If $\srk_M (\Omega_i^{e_0}(R))>0$ for some $e_0\in \N$ then 
\[
\liminf_{e\to \infty}\frac{\srk_M (\Omega_i^e(R))}{\rank(F^e_*R)}>0.
\]
\end{lemma}

\begin{proof}
Suppose $M$ is a direct summand of $\Omega_i^{e_0}(R)$. Observe that $F^{e+e_0}_*R$ has a direct summand $F^{e_0}_*R^{\oplus a_{e}}$. It readily follows that $\Omega_{e+e_0}(R)$ has $\Omega_{e_0}(R)^{\oplus a_e(R)}$ as a direct summand and therefore $\srk_M (\Omega_{e+e_0})\geq a_e(R)$. In particular,
\[
\liminf_{e\to \infty}\frac{\srk_M (\Omega_{e}(R))}{\rank(F^e_*R)}\geq \frac{s(R)}{\rank(F^{e_0}_*R)}>0.
\]
\end{proof}

We are now prepared to prove Theorem~\ref{Main equimultiplicity theorem} for Frobenius Betti numbers and Frobenius Euler characteristics. We first present a proof of Theorem~\ref{Main equimultiplicity theorem} for Frobenius Betti numbers.

\begin{theorem}\label{beti numbers under localization}
Let $(R,\fm)$ be an $F$-finite strongly $F$-regular local ring of prime characteristic $p>0$ and $P\in \Spec R$. Then for each integer $i\geq 0$, $\beta_i^F(R)=\beta_i^F(R_P)$ if and only if $\beta_i^e(R)=\beta_i^e(R_P)$ for every $e\in \N$.
\end{theorem}

\begin{proof} Clearly if $\beta_i^e(R)=\beta_i^e(R_P)$ for every integer $e$ then $\beta_i^F(R)=\beta_i^F(R_P)$. Suppose there exists an integer $e_0$ such that $\mu(\Omega_i^{e_0}(R_P))<\mu(\Omega_i^{e_0}(R))$. For each $e\in \N$ let $b_e=\srk_{\Omega_i^{e_0}(R)}(\Omega^e_i(R))$. Then we can write $\Omega^e_i(R)\cong \Omega^{e_0}_{i}(R)^{\oplus b_e}\oplus M_e$. Localizing at $P$, 
\[
\Omega_i^e(R)_P\cong \Omega_i^e(R_P)\oplus F_P
\]
where $F_P$ is a free $R_P$-module. It readily follows that
\[
\mu(\Omega_i^e(R_P))\leq \mu(\Omega_i^e(R)_P)\leq \mu(\Omega_i^e(R))-b_e.
\]
Therefore $\beta_i^F(R_P)\leq \beta_i^F(R)-\liminf_{e\to \infty}\frac{b_e}{\rank(F^e_*R)}$ which is strictly less then $\beta_i^F(R)$ by Lemma~\ref{summands of syzygies lemma}.
\end{proof}

Following the proof of Theorem~\ref{Theorem Hilbert--Kunz the same}
we recover \cite{AberbachLi} for strongly $F$-regular rings.

\begin{corollary}
Let $(R,\fm)$ be an $F$-finite strongly $F$-regular local ring of prime characteristic $p>0$.
Then for each integer $i\geq 0$, $\beta_i^F(R)=0$ if and only if $R$ is a regular local ring.
\end{corollary}
\begin{proof}
By Lemma~\ref{projective dimension lemma}, $\beta_i^F(R_{(0)}) = 0 =  \beta_i^F(R)$. Therefore $\beta_i^e(R) = \beta_i^e (R_{(0)}) = 0$ and the claim follows from Lemma~\ref{projective dimension lemma}.
\end{proof}

Finally, we complete our proof of Theorem~\ref{Main equimultiplicity theorem} by establishing an equimultiplicity criterion for Frobenius Euler characteristic.

\begin{theorem}\label{Euler characteristic under localization}
Let $(R,\fm)$ be an $F$-finite strongly $F$-regular local ring of prime characteristic $p>0$ and $P\in \Spec R$. Then for each integer $i\geq 0$, $\chi_i^F(R)=\chi_i^F(R_P)$ if and only if $\chi_i^e(R)=\chi_i^e(R_P)$ for every $e\in \N$.
\end{theorem}

\begin{proof} Without loss of generality we may assume $R$ is not regular. By Lemma~\ref{rank of syzygy lemma}, $\chi_i^e(R)=\chi_i^e(R_P)$ if and only if $\rank(\Omega^e_{i+1}(R))=\rank(\Omega^e_{i+1}(R_P))$ and $\chi_i^F(R)=\chi_i^F(R_P)$ if and only if 
\[
\lim_{e\to \infty}\frac{\rank(\Omega_i^e(R))}{\rank(F^e_*R)}=\lim_{e\to \infty}\frac{\rank(\Omega_i^e(R_P))}{\rank(F^e_*R)}.
\]
Suppose there exists an integer $e_0$ such that $\rank(\Omega_{i+1}^{e_0}(R))\not = \rank(\Omega_{i+1}^{e_0}(R_P))$. Therefore  $\Omega_{i+1}^{e_0}(R)_P$ has a nonzero free summand. 
Let $b_e=\srk_{\Omega_{i+1}^{e_0}(R)}(\Omega_{i+1}^e(R))$, by Lemma~\ref{summands of syzygies lemma} $\liminf\limits_{e\to \infty }\frac{b_e}{\rank(F^e_*R)}>0$. Then for each integer $e\in \N$ the $R_P$-module $\Omega_{i+1}^e(R)_P$ contains a free summand of rank $b_e$. In particular, we have that $\rank(\Omega_{i+1}^e(R_P))\leq \rank(\Omega_{i+1}^e(R)_P)-b_e$. Therefore
\[
\lim_{e\to \infty}\frac{\rank(\Omega_{i+1}^e(R_P))}{\rank(F^e_*R)}\leq \lim_{e\to \infty}\frac{\rank(\Omega_{i+1}^e(R))}{\rank(F^e_*R)}-\liminf_{e\to \infty}\frac{b_e}{\rank(F^e_*R)}<\lim_{e\to \infty}\frac{\rank(\Omega_{i+1}^e(R))}{\rank(F^e_*R)}.
\]
\end{proof}

As with $F$-signature, Hilbert--Kunz multiplicity, and Frobenius Betti numbers, we now know that Frobenius Euler characteristic can be used to detect regular rings, provided we know the ring being studied is strongly $F$-regular.

\begin{theorem}
Let $(R,\fm)$ be an $F$-finite strongly $F$-regular local ring of prime characteristic $p>0$. The following are equivalent:
\begin{enumerate}
\item $R$ is a regular local ring;
\item $\chi^e_i(R)=(-1)^i\rank(F^e_*R)$ for every $e\in \NN$;
\item $\chi^e_i(R)=(-1)^i\rank(F^e_*R)$ for some $e\in \NN$;
\item $\chi_i^F(R)=(-1)^i$.
\end{enumerate}
\end{theorem}

\begin{proof}
The equivalence of $(1), (2)$, and $(3)$ is the content of Lemma~\ref{lower bound for Euler characteristic} and $(4)$ is trivially implied by condition $(2)$. Now, an argument with the 
generic point as in Theorem~\ref{theorem signature 1 iff regular} shows that $(4)$ implies $(2)$ by Theorem~\ref{Euler characteristic under localization}.
\end{proof}

\section{An associativity formula for $F$-signature}\label{Section Associativity}

Our proof of Theorem~\ref{Main Theorem Associativity type formula} begins with two technical lemmas.

\begin{lemma}\label{fsig monotone}
Let $(R,\mf m)$ be an $F$-finite local ring of prime characteristic $p>0$. Suppose that $I\subset R$ is an ideal such that $R/I$ is Cohen-Macaulay of dimension $h$ and $x_1, x_2, \ldots, x_h$ a parameter sequence on $R/I$. Then for all sequences of natural numbers $n_1,n_2,\ldots, n_h$ we have that 
\begin{enumerate}
\item \[
\length \left(\frac{I_e(I+ (x_1^{n_1},x_2^{n_2},\ldots, x_h^{n_h}))}{I_e(I+ (x_1^{n_1+1},x_2^{n_2},\ldots, x_h^{n_h}))}\right)\geq \length \left(\frac{I_e(I+ (x_1^{n_1-1},x_2^{n_2},\ldots, x_h^{n_h}))}{I_e(I+ (x_1^{n_1},x_2^{n_2},\ldots, x_h^{n_h}))}\right)
\]
and
\item \[
\frac{1}{n_1+1}\lambda\left(\frac{R}{I_e(I+(x_1^{n_1+1},x_2^{n_2},\ldots, x_h^{n_h}))}\right)\geq \frac{1}{n_1}\lambda\left(\frac{R}{I_e(I+(x_1^{n_1},x_2^{n_2},\ldots, x_h^{n_h})}\right).
\]
\end{enumerate}
\end{lemma}
\begin{proof}
We may pass to $I' = I + (x_2^{n_2}, \ldots, x_h^{n_h})$
and assume that $\dim R/I = 1$.

We claim there exist short exact sequences
\[
0\to \frac{I_e(I+(x^{n-1}))}{I_e(I+(x^n))}\xrightarrow{\cdot x^{p^e}}\frac{I_e(I+(x^{n}))}{I_e(I+(x^{n+1}))}\to \frac{I_e(I+(x^{n}))}{I_e(I+(x^{n+1}))+x^{p^e}I_e(I+(x^{n-1}))}\to 0.
\]
Observe that if such short exact sequences exist then the first inequality is obvious since the length of the left piece of a short exact sequence is no more than the length of the middle term. 
The second inequality is equivalent to 
the inequality
\[
n \length \left (\frac {I_e(I + (x^{n}))}{I_e (I + (x^{n+1}))} \right)
= 
n\left (\length \left ( \frac{R}{I_e(I + (x^{n + 1}))} \right) - \length \left ( \frac{R}{I_e(I + (x^n))} \right) \right) \geq \length (R/I_e(I + (x^n))),
\]
an inequality which follows from the first since we can filter $\length (R/I_e(I + (x^n)))$ as 
\[
\length (R/I_e(I + (x^n))) = 
\sum_{i = 0}^{n - 1} \length \left (\frac{I_e(I+ (x^{i}))}{I_e(I + (x^{i + 1}))} \right).
\]

To show that the above short exact sequences exist, we first notice that 
\[
x^{p^e}I_e(I+(x^{n-1}))\subseteq I_e(xI+(x^{n})) \subseteq I_e(I+(x^{n})).
\]
Indeed, if $u\in I_e(I+(x^{n-1}))$ and $\varphi \in \Hom_R(F^e_*R,R)$ then \[\varphi(F^e_*x^{p^e}u)=x\varphi(F^e_*u)\in x(I+(x^{n-1}))\subseteq (I+(x^n)).\] 
Therefore there are right exact sequences
\[
\frac{I_e(I+(x^{n-1}))}{I_e(I+x^{n}))}\xrightarrow{\cdot x^{p^e}}\frac{I_e(I+(x^{n}))}{I_e(I+(x^{n+1}))}\to \frac{I_e(I+(x^{n}))}{I_e(I+(x^{n+1}))+x^{p^e}I_e(I+(x^{n-1}))}\to 0.
\]
To show injectivity of the first map observe that an element $u\in I_e(I+(x^{n-1}))$ satisfies $x^{p^e}u\in I_e(I+(x^{n+1}))$ if and only if $u\in I_e(I+(x^{n+1})):x^{p^e}$. By (7) of Lemma~\ref{Basic properties} we have that 
\[
I_e(I+(x^{n_1+1})):x^{p^e}=I_e((I+(x^{n+1})):x)=I_e(I+(x^{n})),
\]
where the second equality follows by standard observations on parameter ideals in the Cohen-Macaulay ring $R/I$.
\end{proof}

The following technical lemma is very much in the spirit of \cite[Theorem~4.3]{PolstraTucker}.

\begin{lemma}\label{Technical lemma}
Let $(R, \mf m)$ be an $F$-finite local domain of prime characteristic $p>0$ and of Krull dimension $d$. Suppose that $I\subset R$ is an ideal such that $R/I$ is Cohen-Macaulay of dimension $h$ and $\underline{x}=x_1,\ldots, x_h$ a parameter sequence on $R/I$. Then there exists a constant $C\in \mathbb{R}$ such that for all $e,n_1,n_2,\ldots, n_h\in \N$
\[
\left|\frac{1}{p^{ed}}\lambda(R/I_e(I+(x_1^{n_1},\ldots, x_h^{n_h}))-\fsig(I+(x_1^{n_1},\ldots, x_h^{n_h}))\right|\leq \frac{Cn_1\cdots n_h}{p^e}.
\]
\end{lemma}
\begin{proof}
Denote by $\underline{N}$ a Cartesian product of natural numbers, $\underline{N}=(n_1,n_2,\ldots, n_h)$, let $N=n_1n_2\cdots n_h$, and for each $\underline{N}$ let $\underline{x}^{\underline{N}}$ be the sequence of elements $x_1^{n_1}, x_2^{n_2},\cdots, x_h^{n_h}.$  We are claiming there exists a constant $C$, depending only on $\length(R/(I+(\underline{x})))$, such that for all $\underline{N}$
\[
\left|\frac{1}{p^{ed}}\length(R/I_e(I+(\underline{x})^{\underline{N}}))-\fsig(I+(\underline{x}^{\underline{N}}))\right|\leq \frac{CN}{p^{e}}.
\]
We will first show that there exists a constant $C$ such that for all $\underline{N}$ and $e\in \N$
\[
\frac{1}{p^{ed}}\length(R/I_e(I+(\underline{x})^{\underline{N}}))\leq \fsig(I+(\underline{x}^{\underline{N}}))+\frac{CN}{p^e}.
\]
The $R$-module $F_*R$ is finitely generated and torsion-free so there exists a short exact sequence
\[
0\to F_*R\xrightarrow{\psi} R^{\oplus \rank(F_*R)}\to T\to 0
\]
where $T$ is a finitely generated torsion $R$-module. By (4) of Lemma~\ref{Basic properties}
\[
\psi(F_*I_{e+1}(I+(\underline{x}^{\underline{N}})))\subseteq I_e(I+(\underline{x}^{\underline{N}}))^{\oplus \rank(F_*R)}
\]
and therefore there are right exact sequences
\[
F_*R/(I_{e+1}(I+(\underline{x}^{\underline{N}}))\xrightarrow{\psi } R^{\oplus \rank(F_*R)}/I_e(I+(\underline{x}^{\underline{N}}))^{\oplus\rank(F_*R)}\to T_{e}\to 0
\]
where $T_{e}$ is the homomorphic image of $T/I_e(I+(\underline{x}^{\underline{N}}))T$. Therefore
\begin{equation}\label{firstinequality}
\rank(F_*R)\length(R/I_e(I+(\underline{x}^{\underline{N}})))\leq \length(F_*R/I_{e+1}(I+(\underline{x}^{\underline{N}})))+\length(T/I_e(I+(\underline{x}^{\underline{N}}))T).
\end{equation}
Suppose that $c\in R$ is a nonzero element which annihilates $T$. Because $(I+(\underline{x}^{N}))^{[p^e]}\subseteq I_e(I+(\underline{x}^{\underline{N}}))$ there exists a surjective map \[
(R/(c, (I+(\underline{x}^{\underline{N}}))^{[p^e]}))^{\oplus \mu (T)}\to T/I_e(I+(\underline{x}^{\underline{N}}))T
\]
and we have that
\[
\length (T/I_e(I+(\underline{x})T))\leq \mu(T)\length(R/(c,(I+(\underline{x}^{N}))^{[p^e]})).
\]
It is well known there exists $C\in \mathbb{R}$, depending only on the ring $R$, such that
\[
\length(R/(c,(I+(\underline{x}^{N})))^{[p^e]})\leq Cp^{(e\dim(R)-1)}\length (R/(I+(\underline{x}^{\underline{N}}))),
\]
see \cite[Proposition~3.3]{PolstraLSC} for example. Because $R/I$ is Cohen-Macaulay we know that 
\[
\length(R/(I+(\underline{x}^{\underline{N}})))=\eh(\underline{x}^{\underline{N}};R/I)=N\eh(\underline{x};R/I)=N\length(R/(I+(\underline{x}))).
\]
If we divide the inequality in \ref{firstinequality} by $\rank(F_*R)p^{(e+1)d}$ we obtain that
\[
\frac{\length(R/I_{e}(I+(\underline{x}^{\underline{N}})))}{p^{ed}}\leq \frac{\length(R/I_{e+1}(I+(\underline{x}^{\underline{N}})))}{p^{(e+1)d}}+\frac{\mu(T)C\length(R/(I+\underline{x}))N}{p^e}
\]
for every $e\in \N$. The constant $\mu(T)C\length(R/(I+\underline{x}))$ has no dependence on $e$ or $N$, so we replace $C$ by this constant and utilize \cite[Lemma~3.5]{PolstraTucker} to obtain that
\[
\frac{\length(R/I_{e}(I+(\underline{x}^{\underline{N}})))}{p^{ed}}\leq \fsig(I+(\underline{x}^{N}))+\frac{2CN}{p^e}.
\]

Obtaining inequalities of the form 
\begin{equation}\label{inequality2}
\fsig(I+(\underline{x}^{N}))\leq \frac{\length(R/I_{e}(I+(\underline{x}^{\underline{N}})))}{p^{ed}}+\frac{CN}{p^e},
\end{equation}
 is almost identical to the above. Begin by examining a short exact sequence of the form
\[
0\to R^{\oplus \rank(F_*R)}\xrightarrow{\psi} F_*R\to T'\to 0
\]
where $T'$ is a torsion $R$-module. By (3) of Lemma~\ref{Basic properties} we have that $\psi(I_e(I+(\underline{x}^{\underline{N}}))^{[p]})\subseteq I_{e+1}(I+(\underline{x}^{\underline{N}}))F_*R$ and so there are right exact sequences
\[
\left(R/I_e(I+(\underline{x}^{\underline{N}}))^{[p]}\right)^{\oplus \rank(F_*R)}\xrightarrow{\psi} F_*R/I_{e+1}(I+(\underline{x}^{\underline{N}}))F_*R\to T'_e\to 0
\]
where $T'_e$ is the homomorphic image of $T'/I_{e+1}(I+(\underline{x}^{\underline{N}}))T'$. The reader is now encouraged to follow the techniques above and the techniques of \cite[Theorem~4.3]{PolstraTucker} to obtain inequalities as described in \ref{inequality2}.
\end{proof}

For the proof of the Theorem~\ref{Main Theorem Associativity type formula}  we recall the following standard result:
if $a_{m,n}$ is a bisequence such that
\begin{itemize}
    \item $\lim\limits_{m,n \to \infty} a_{m,n}$ exists, and
    \item $\lim\limits_{n \to \infty} a_{m,n}$ exists for all $m$,
\end{itemize}
then $\lim\limits_{m,n \to \infty} a_{m,n} = \lim\limits_{m \to \infty} \lim\limits_{n \to \infty} a_{m,n}$.

\begin{theorem}\label{Associativity formula for $F$-signature}
Let $(R, \mf m)$ be an $F$-finite  local ring of prime characteristic $p>0$ and of Krull dimension $d$. Suppose that $I\subseteq R$ is an ideal such that $R/I$ is Cohen-Macaulay of dimension $h$ and $\underline{x}=x_1,\ldots, x_h$ a parameter sequence on $R/I$. Then 
\[
\lim_{n_1,\ldots,n_h\to \infty}\frac{1}{n_1\cdots n_h}\fsig(I,(x_1^{n_1},\ldots, x_h^{n_h}))=\sum_{P} \eh(x_1,\ldots,x_h;R/P)\fsig(IR_P),
\]
where the sum is taken over all prime ideals $P\supseteq I$ such that $\dim(R/I)=\dim(R/P)$.
\end{theorem}

\begin{proof}
Lemma~\ref{Technical lemma} allows us to swap limits and identify 
\begin{align*}
\lim_{n_1,\ldots,n_h\to \infty}\frac{1}{n_1\cdots n_h}&\fsig(I,(x_1^{n_1},\ldots, x_h^{n_h}))\\
&=\lim_{n_1,\ldots,n_h\to \infty}\lim_{e\to \infty}\frac{1}{n_1\cdots n_h p^{ed}} \lambda(R/I_e(I+(x_1^{n_1},\ldots, x_h^{n_h}))\\
&=\lim_{e\to \infty}\lim_{n_1,\ldots,n_h\to \infty}\frac{1}{n_1\cdots n_h p^{ed}} \lambda(R/I_e(I+(x_1^{n_1},\ldots, x_h^{n_h})).
\end{align*}
Furthermore, by Lemma~\ref{fsig monotone}
\begin{align*}
\lim_{n_1,\ldots,n_h\to \infty}\frac{\length (R/I_e(I +(x_1^{n_1},\ldots, x_h^{n_h})))}{n_1\cdots n_h} 
&= \sup_{n_1,\ldots,n_h}\frac{\length (R/I_e(I +(x_1^{n_1},\ldots, x_h^{n_h})))}{n_1\cdots n_h}
\\&= \sup_n \frac{\length (R/I_e(I +(x_1^{n},\ldots, x_h^{n})))}{n^h}.
\end{align*}
We prove the theorem by induction on $h$. Let us start with the 
case of $h = 1$.

To prove the claim, let us introduce an auxiliary bisequence that will link the two sides of the formula together. 
\begin{claim}
For each pair of natural numbers $n,m\in \mathbb{N}$ let 
\[
a_{n, m}=\length \left( R/(I_e(I+(x^{n+m}))+(x^{np^e})) \right ).
\]
Then the bisequence $a_{n,m}$ satisfies the following properties:
\begin{enumerate}
\item $a_{n, 0} = \length (R/I_e(I +(x^{n})))$;
\item $\lim\limits_{m\to \infty} a_{n, m} = \length (R/(I_e(I) + (x^{np^e})))$;
\item $a_{n, m} = a_{n + m, 0} - a_{m,0}$.
\end{enumerate}
\end{claim}
\begin{proof}
The first two properties are immediate from the definition.

For the third formula we first recall that 
if $J$ is an ideal and $x \notin J$ then 
$\length (R/(J,x)) = \length (R/J) - \length (R/J:x)$.
Applying this to $J = I_e(I + (x^{n+m}))$ and $x^{np^e}$
we obtain by (7) of Lemma~\ref{Basic properties} that
\begin{align*}
a_{n,m} &= \length (R/I_e(I + (x^{n+m}))) - 
\length (R/I_e(I + (x^{n+m})):x^{np^e})\\
&= a_{n+m, 0} - \length (R/I_e((I + (x^{n+m})):x^n))\\
&= a_{n+m, 0} - a_{m, 0}.
\end{align*}
\end{proof}

Recall that for any bisequence $\sup_n \sup_m a_{n,m} = \sup_{n,m} a_{n,m} = \sup_m \sup_n a_{n,m}$.
By definition, the sequence $a_{n,m}$ is increasing in $m$, so 
by the claim
\[
\sup_{n} \sup_m \frac{a_{n,m}}{n} = 
\sup_n \frac 1n \length (R/I_e(I +(x^{n})))
= \sum_{P} \eh(x, R/P) \length (R_P/I_e(I)R_P),
\]
because $x$ is a regular element modulo $I_e(I)$ by (9) of
Lemma~\ref{Basic properties}.
On the other hand, by Lemma~\ref{fsig monotone} 
$a_{n, 0}/n$ is an increasing function in $n\in \NN$, so 
the claim also shows that 
\[
\sup_{n} \frac{a_{n,m}}{n}
= \sup_{n} \frac{a_{n+m, 0} - a_{m, 0}}{n}
= \lim_{n\to \infty} \frac{a_{n, 0}}{n}.
\]
Thus
\begin{align*}
\sum_{P} \eh(x, R/P) \length (R_P/I_e(I)R_P)
= \sup_{n} \sup_m \frac{a_{n,m}}{n}
= \sup_m \lim_{n\to \infty} \frac{a_{n, 0}}{n}
= \lim_{n\to \infty} \frac{\length (R/I_e(I +(x^{n})))}{n},
\end{align*}
which proves the theorem in the case $h = 1$
after passing to the limit as $e \to \infty$. 

For $h\geq 2$ we may first consider the ideal 
$I' = I + (x_1^n, \ldots, x_{h-1}^n)$ and 
get that
\[
\lim_{m \to \infty}
\frac{\length (R/I_e(I' + (x_h^m)))}{m}
= \sum_Q \eh(x_h; R/Q) \length (R_Q/I_e(I')R_Q),
\]
where $Q$ varies through the prime ideals $P$ containing  
$I_e(I')$ such that $\dim R_Q/PR_Q = \dim R_Q/I_e(I')R_Q$.
By induction,  
\[
\lim_{n \to \infty}
\frac{\length (R_Q/I_e(I')R_Q)}{n^{h-1}}
= \sum_P \eh(x_1, \ldots, x_{h-1}; R_Q/PR_Q) \length (R_P/I_e(I)R_P),
\]
where $P$ varies through the prime ideals $P$ containing $I_e(I)$
such that $\dim R_Q/PR_Q = \dim R_Q/I_e(I)R_Q$.
Thus 
\begin{align*}
\lim_{n \to \infty}
&\frac{\length (R/I_e(I +(x_1^{n},\ldots, x_h^{n})))}{n^h}\\
&= \lim_{n \to \infty} \lim_{m \to \infty} 
\frac{\length (R/I_e(I' + (x_h^m)))}{mn^{h-1}}\\
&= \sum_Q \sum_P \eh(x_h; R/Q) \eh(x_1, \ldots, x_{h-1}; R_Q/PR_Q) \length (R_P/I_e(I)R_P).
\end{align*}
The theorem follows by changing the order of summation 
and using the associativity formula for parameter ideals (\cite[Theorem~1]{Lech}), 
see the proof of \cite[Theorem~4.9]{equi}.
\end{proof}

\bibliographystyle{alpha}
\bibliography{deg}

\end{document}